\documentclass[12pt]{article}
\usepackage{a4wide}
\usepackage{amsmath,amsfonts,amssymb,amsthm}
\usepackage[percent]{overpic} 
\usepackage{color}

\usepackage{hyperref}

\definecolor{modra3}{rgb}{.1,.0,.4}
\hypersetup{colorlinks=true, linkcolor=modra3, urlcolor=modra3, citecolor=modra3, pdfpagemode=UseNone, pdfstartview=} 

\usepackage{graphicx}
\usepackage{epsfig}
\usepackage{verbatim}
\usepackage{pstool}
\usepackage{psfrag}

\theoremstyle{plain}
\newtheorem{theorem}{Theorem} 
\newtheorem{lemma}[theorem]{Lemma}
\newtheorem{proposition}[theorem]{Proposition}

\newtheorem{corollary}[theorem]{Corollary}

\theoremstyle{remark}
\newtheorem{remark}[theorem]{Remark}
\newtheorem*{claim_}{Claim}

\begin{document}

\title{Spiraling and Folding: The Topological View\thanks{An earlier version of this paper appeared in the proceedings of
the 19th Canadian Conference on Computational Geometry~\cite{SSS07a}.}}

\author{Jan Kyn\v{c}l\thanks{Charles University, Prague, Czech Republic; \texttt{kyncl@kam.mff.cuni.cz}. Supported by project 22-19073S of the Czech Science Foundation (GA\v{C}R) and by Charles University project UNCE/SCI/004. } 
\and 
Marcus Schaefer\thanks{DePaul University, Chicago, IL 60604, USA; \texttt{mschaefer@cdm.depaul.edu} }
\and 
Eric Sedgwick\thanks{DePaul University, Chicago, IL 60604, USA; \texttt{esedgwick@cdm.depaul.edu} } 
\and 
Daniel \v{S}tefankovi\v{c}\thanks{University of Rochester, Rochester, NY 14627, USA; \texttt{stefanko@cs.rochester.edu}} 
} 

\maketitle


\begin{abstract}
For every $n$, we construct two curves in the plane that intersect  
at least $n$ times and do not form spirals. The construction is
in three stages: we first exhibit closed curves on the
torus that do not form double spirals, then arcs on the torus that
do not form spirals, and finally pairs of planar arcs that do not
form spirals.  These curves provide a counterexample to a proof of 
Pach and T\'{o}th concerning string graphs.
\end{abstract}

\section{Introduction}

In a surface, draw a pair of simple curves with a large number of
intersections. Must the drawing contain any particular
substructures?  And if so, can these structures be used to
simplify the drawing?   Affirmative answers to these questions can
be used to bound the complexity of certain types of drawings, such as those
yielding string graphs.   A \emph{string graph} is the intersection
graph of curves in the plane or some other surface.  The recognition problem for 
string graphs is an old problem~\cite{B59,G76,S66}, \cite[Research problem 1]{AHM78}
that was solved in the planar case~\cite{SS04} by establishing an exponential bound on the
complexity of such drawings.\footnote{The problem is decidable on arbitrary surfaces, as shown in~\cite{SSS03}.} Pach and T\'{o}th~\cite{PT01} also claimed
a solution to the string graph problem, but as we will see in Section~\ref{sec:counter}
their proof contains a serious gap. What is so tricky about this problem?

Starting with a drawing of a pair of simple curves, one can increase the number of
intersections by taking a part of one curve and dragging it over
another curve. This creates a \emph{bigon}, a disk region bounded by exactly one subarc of each curve, whose interior is disjoint from the curves; see Figure~\ref{fBigon}. Of course, there is an obvious simplification of
such a drawing. We call a drawing \emph{reduced} if it contains no bigon. It is not hard to create a reduced drawing with
many intersections, so the question arises: do
such drawings contain some other type of structure?

\begin{figure}
\begin{center}
\includegraphics{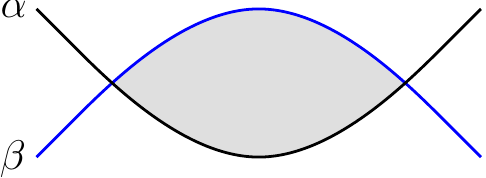}
\caption{A bigon formed by $\alpha$ and $\beta$.}
\label{fBigon}
\end{center}
\end{figure}

One candidate for such a structure is a \emph{fold}, pictured on the right in Figure~\ref{fFold}. 
The drawing of the fold can be made into a reduced drawing by adding an endpoint to the center bigon or by altering the topology of the underlying surface in the bigon. 
Kratochv\'{\i}l and Matou\v{s}ek~\cite{KM91} use folds to create string graphs requiring $\Omega(2^n)$ intersections in any realization.

\begin{figure}
\begin{center}
\includegraphics{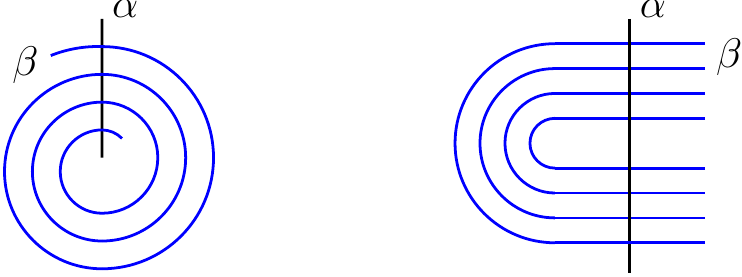}
\caption{Left: A spiral formed by $\alpha$ and $\beta$. Right: A fold formed by $\alpha$ and $\beta$. The vertical arc is a part of the curve $\alpha$, 
all the other arcs are parts of the curve $\beta$.} 
\label{fFold}

\end{center}
\end{figure}

The focus in the current paper is a third structure, the \emph{spiral}, pictured on the left in Figure~\ref{fFold}. Spirals can be used to simplify certain
drawings.  For example, Pach and T\'{o}th~\cite{PT01} demonstrate that spirals can be used to reduce the complexity of string graphs in the plane. However, in this and an earlier
companion paper~\cite{SSS11}, we exhibit drawings with an arbitrary
number of intersections and no spirals. The companion paper shows
that certain drawings in the torus are spiral-free by analyzing
intersection words. Here we use topological techniques to construct
spiral-free drawings both in the torus and in the plane. All the
drawings are reduced and the torus drawings are also fold-free.
Thus, eliminating spirals and folds is not sufficient to bound the complexity of drawings in the torus (or surfaces of higher genus).
And, as we show in Section~\ref{sec:counter}, spiral-free planar drawings are a counterexample to Pach and T\'{o}th's claim that wide folds imply the existence of a spiral (or a bigon)~\cite{PT01}.

In the next section, we define the basic concepts including spirals and train tracks.  In Section~\ref{sec:noSpiralsTorus}, we use train tracks to describe spiral-free drawings in the torus.  Then, in Section~\ref{sec:noSpiralsPlane}, we show how to use the spiral-free torus drawings to construct spiral-free drawings in the plane.

\section{Curves in Surfaces}

In this section we introduce basic terminology for curves in surfaces, show how they can be represented using train tracks and develop basic facts on spirals.

\subsection{Curves}

We refer the reader to~\cite{Code17_computational,farb2012primer,Stillwell93} for a basic reference on the topology of surfaces.  
Here we consider \emph{surfaces} to be compact and orientable.  We allow surfaces to have boundary components (which can be obtained by puncturing a closed surface). 
A \emph{disk} is a surface that is homeomorphic to the closed unit disk, that is, the closure of the region bounded by the unit circle in the plane.   An \emph{annulus} is a surface that is homeomorphic to the closure of the region bounded by two concentric circles in the plane. 

A \emph{simple arc} in a surface $F$ is an image of an injective continuous map from the closed unit interval $[0,1]$ to $F$. 
A \emph{simple closed curve} or a \emph{simple loop} in a surface $F$ is an image of an injective continuous map from the unit circle to $F$.
A \emph{simple curve} is a simple arc or a simple closed curve. If $\alpha$ is a simple arc, we denote by $\partial\alpha$ the two-element set of its endpoints. If $\alpha$ is a simple closed curve, then $\partial\alpha=\emptyset$. 
In other words, a simple curve in $F$ is a connected 1-manifold embedded in $F$ (with no self-intersections).  We will refer to simple curves shortly as \emph{curves}, and we use the term  \emph{multicurve} to refer to the union of finitely many pairwise disjoint simple curves.    A simple closed curve $\alpha$ is \emph{essential} in a compact surface $F$ if $\alpha$ does not bound a disk in $F$, and $\alpha$ is not homotopic\footnote{Two curves are {\em homotopic} if they can be continuously deformed into each other.} to any boundary component of $F$.

A curve (or multicurve) $\alpha$ in a surface $F$ with boundary is \emph{properly embedded} in $F$ if it is simple and it meets the boundary of $F$ exactly in its endpoints (if any); that is, $\alpha \cap \partial F = \partial \alpha$.

An \emph{isotopy} of a curve is a continuous deformation (homotopy) of that curve that, at each point in time, is a proper embedding of the curve.
For properly embedded simple arcs we also insist that the isotopy fixes the endpoints.  Two curves are \emph{isotopic} if there is an isotopy between them.

In drawings of more than one curve, we assume that the curves are in general position in the sense that they do not touch or overlap; if two curves intersect, they have to do so transversely.  Let  $|\alpha \cap \beta|$  denote the number of intersection points between curves $\alpha$ and $\beta$.

\subsection{Bigons}
Let $D$ be a disk and $\alpha$ and $\beta$ be arcs that are properly embedded in $D$. The arcs $\alpha$ and $\beta$ \emph{form a bigon $B$ in $D$} if they intersect precisely twice in $D$.  In this case, the \emph{bigon} $B$ will be the closed disk region bounded by subarcs of $\alpha$ and $\beta$.
We say that a pair of curves  $\alpha$ and $\beta$ \emph{form a bigon} in a surface $F$ if there is a disk $D \subset F$ such that subarcs of $\alpha$ and $\beta$ form a bigon in $D$.

If $\alpha$ and $\beta$ do not bound a bigon, we will say that their intersection is
\emph{reduced}.  We say that a drawing consisting of a set of curves is \emph{reduced} if all pairs of curves in the drawing have reduced intersection.

\subsection{Folds}
 Let $A$ be an annulus and $\alpha,\beta$ be two curves. Let $\alpha_0$, $\beta_1, \beta_2, \dots, \beta_k$ be arcs properly embedded in $A$ with distinct endpoints on the outer boundary curve of $A$, such that $\alpha \cap A = \alpha_0$ and $\beta \cap A = \bigcup_{i=1}^k \beta_i$. The arcs $\alpha$ and $\beta$ \emph{form a fold of width $k$ in $A$} if they have reduced
intersection (no bigons) in $A$ and every $\beta_i$ intersects $\alpha_0$ in exactly two points
in $A$.  See Figure~\ref{fig_fold}.
If a boundary curve of the annulus defining a fold bounds a disk in the total surface $F$, then this forces a  bigon, so the drawing is not reduced.

\begin{figure}
\begin{center}
\includegraphics{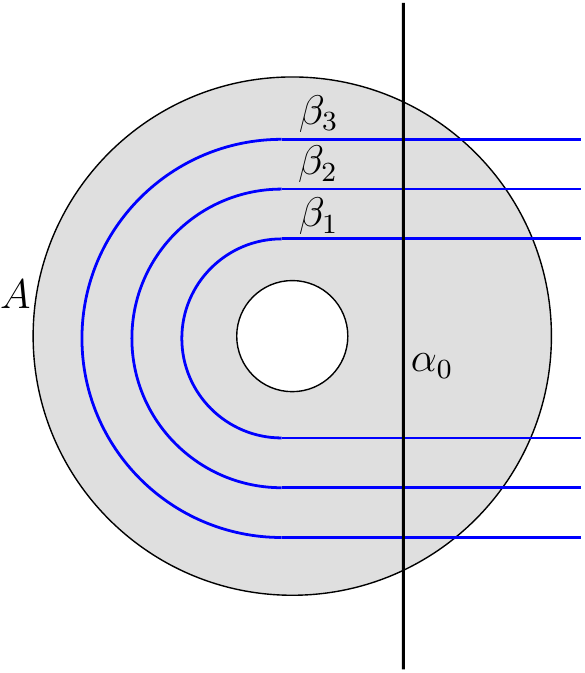}
\caption{A fold of width $3$ formed by $\alpha$ and $\beta$ in $A$.} 
\label{fig_fold}
\end{center}
\end{figure}

\subsection{Spirals}
Let $A$ be an annulus and $\beta$ and $\gamma$ be \emph{spanning
arcs} for $A$, that is, arcs properly embedded in $A$ and with
endpoints in opposite boundary curves of $A$.  The arcs $\beta$ and
$\gamma$ \emph{form a $d$-spiral in $A$} if they have distinct endpoints, have reduced
intersection (no bigons) in $A$, and intersect in $d+2$ points
in $A$.  See Figure~\ref{fSpiral}. The dashed lines indicate the possibility that $\beta$ and $\gamma$ are subarcs of larger curves, and  other subarcs of the larger curves may also be present in the annulus $A$. If $d \ge 1$, we say that {\it
$\beta$ and $\gamma$ form a spiral in $A$}.

A pair of curves $\beta$ and $\gamma$ in a surface $F$
form a \emph{$d$-spiral in $F$ with respect to an annulus  $A \subset F$} if subarcs of $\beta$ and $\gamma$ form a
$d$-spiral in $A$. Again, if $d \ge 1$, we say that \emph{$\beta$
and $\gamma$ form a spiral in $F$ (with respect to $A$)}.

\begin{figure}
\begin{center}
\includegraphics{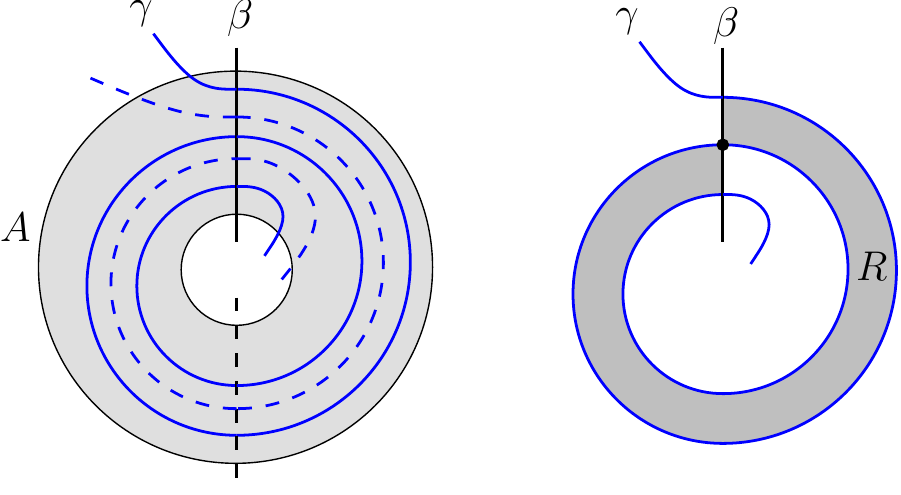}
\caption{$\beta$ and $\gamma$ form a $1$-spiral. The dashed lines indicate that other subarcs may be present.} \label{fSpiral}
\end{center}
\end{figure}

Note that $\beta$ and $\gamma$ forming a spiral in $F$ is equivalent
to the existence of a ``singular'' rectangle $R \subset F$, a region formed by one subarc of
$\beta$ and one subarc of $\gamma$, whose interior is homeomorphic to an open disc, but its boundary path meets itself in exactly one point; see
Figure~\ref{fSpiral}.

\begin{lemma} 
\label{lemLEQ2} 
Suppose that $\beta$, $\beta'$, $\gamma$ and
$\gamma'$ are spanning arcs for an annulus $A$ with pairwise reduced intersection. 
If $\beta \cap \beta' = \gamma \cap \gamma' = \emptyset$, then $\big| | \beta \cap  \gamma | - | \beta \cap \gamma' | \big|
\le 1$ and $\big| |\beta \cap  \gamma | - | \beta' \cap \gamma' | \big| \le 2$.
 \end{lemma}
 
The lemma uses $|\cdot|$ in two different meanings: applied to sets it is the cardinality of the set, applied to numbers it is the absolute value.

\begin{proof} Since $\gamma$ and $\gamma'$ are disjoint they
cut the annulus $A$ into two
rectangles, $R_1$ and $R_2$ as illustrated in Figure~\ref{fR1R2}. The arc $\beta$
has one endpoint on the top of a rectangle, say $R_1$, and another
endpoint on the bottom of a rectangle, either $R_1$ or $R_2$.  Since
$\beta$ has reduced intersection with both $\gamma$ and $\gamma'$,
there are no bigons in either rectangle.  Aside from the two arcs
meeting the top and bottom of the rectangles, every other subarc of
$\beta$ in a rectangle is horizontal, joining opposite sides.   If
the top and bottom endpoints of $\beta$ are both in $R_1$, then
all subarcs of $\beta$ in $R_2$ are horizontal, implying that $| \beta \cap \gamma | = | \beta\cap \gamma' |$. Otherwise, the top endpoint of $\beta$ is
in $R_1$ and the bottom endpoint in $R_2$. In this case, we have $\big| | \beta,
\gamma | - | \beta\cap \gamma' | \big|=1$.

\begin{figure}
\begin{center}
\includegraphics[height=1.5in]{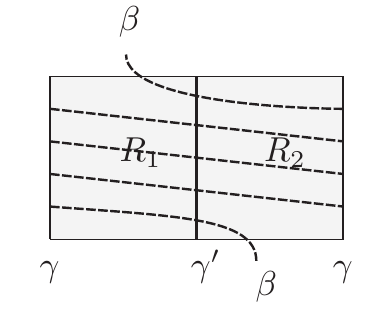}
\caption{$\beta$ passing through $R_1$ and $R_2$} 
\label{fR1R2}
\end{center}
\end{figure}

The second inequality follows by two applications of the first:
$\big| | \beta\cap \gamma | - | \beta'\cap \gamma' | \big| = \big| | \beta\cap \gamma
| - | \beta\cap \gamma' | + | \beta\cap \gamma' | - | \beta'\cap \gamma' |
\big| \le \big| | \beta\cap \gamma | - | \beta\cap \gamma' | \big| + \big| | \beta \cap
\gamma' | - | \beta'\cap \gamma' | \big| \le 2$.
\end{proof}

\begin{lemma} \label{lemAnnulusBound}
 Suppose that the curves $\beta$ and $\gamma$ have reduced intersection and form a
$d$-spiral in a surface $F$ with respect to an annulus $A$. Let $\alpha$ be one of the two boundary curves
of $A$. Let $k = |\beta \cap \alpha|$ and $l = |\gamma \cap \alpha|$. Then
\[
|\beta \cap \gamma| \ge dkl+k+l.
\]
\end{lemma}

\begin{proof}
The curve $\beta$ meets the annulus $A$ in at least $k$ spanning arcs, $\beta_1, \ldots, \beta_k$.
Similarly, $\gamma$ meets $A$ in at least $l$ spanning arcs, $\gamma_1, \ldots, \gamma_l$.

By the definition of spiral, some pair  ($\beta_1,
\gamma_1)$, has $d+2$ intersections in $A$. By Lemma~\ref{lemLEQ2}, pairs of the form ($\beta_i, \gamma_1)$ where $i\neq1$ and
$(\beta_1, \gamma_j)$ where $j\neq 1$ each have at least $d+1$
intersections in $A$. Also, by Lemma~\ref{lemLEQ2}, pairs of the
form $(\beta_i,\gamma_j)$ with $i \neq 1, j \neq 1$ have at least $d$
intersections in $A$. The total number of intersections between
$\beta$ and $\gamma$ in $A$ is therefore at least
\[
(d+2) + (d+1)(k-1 + l-1) + d(k-1) (l-1) = dkl + k+l.
\]
This serves as lower bound on $|\beta \cap \gamma|$, the total number of
intersections of $\beta$ and $\gamma$.
\end{proof}

Lemma \ref{lemAnnulusBound} establishes a lower bound on $|\beta \cap \gamma|$ but there could be many more intersections for two reasons:  First, $\beta$ and $\gamma$ may meet elsewhere in the surface $F$, the lemma only counts intersections in $A$. Second, the lemma allows for \emph{slippage}:  While the  subarcs of $\beta$ and $\gamma$ that form the $d$-spiral intersect $d+2$ times, others pairs of subarcs from $\beta$ and $\gamma$ may slip and can have as few as $d$ intersections in $A$ (see Lemma \ref{lemLEQ2}). For example, if you take 2 close parallel copies of each of the arcs that form a 1-spiral, there will be $12=4 \cdot 3$ intersections in $A$, while the lemma only guarantees $8 = 1 \cdot 2 \cdot 2 + 2 + 2$. The lower bound of 8 can be realized by letting the second arcs slip, see the left side of Figure \ref{fSpiral}.

\begin{lemma}
\label{lemConcentricSpirals}
Let $\beta$ and $\gamma$ be properly embedded multicurves that have reduced intersection in $F$. Let  $\beta_1, \beta_2$ be subarcs of $\beta$ and $\gamma_1, \gamma_2$ subarcs of $\gamma$ so that for some $d_1,d_2 \ge 1$,
\begin{itemize}
\item $\beta_1$ and $\gamma_1$ form a $d_1$-spiral in $F$ with respect to an annulus $A_1$, and
\item $\beta_2$ and $\gamma_2$ form a $d_2$-spiral in $F$ with respect to an annulus $A_2$.
\end{itemize} 
If the defining annuli $A_1$ and $A_2$ are
disjoint in $F$ but have boundary curves that are isotopic in $F$ then $\beta$ and $\gamma$ form a $(d_1+d_2)$-spiral in $F$.
\end{lemma}

\begin{proof}
Since the boundaries of the two annuli are parallel, there is another annulus between $A_1$ and $A_2$.  In this situation, we have concentric spirals such as those pictured
in Figure~\ref{fConcentricSpirals}. 
A priori, the spirals could wind in opposite directions. It will be a consequence of the assumptions and proof that this actually does not occur, see Remark~\ref{rem:counterwind}.  Regardless, call the union of all three annuli $A$.

Let $\beta_1'$ and $\gamma_1'$ be the arcs obtained by extending $\beta_1$ and $\gamma_1$ along $\beta$ and $\gamma$ until they exit $A$. (The arcs cannot close up because $\beta_1$ and $\gamma_1$ have an endpoint on the boundary of $A$, so they must exit as there are no endpoints in the interior of $A$.) Then $\beta_1'$ and $\gamma_1'$ must be spanning arcs for $A$:
Suppose, for contradiction, that one of them, say $\beta_1'$, does not span $A$ and instead has both endpoints on the same boundary component of $A$; see Figure~\ref{fMustSpan}.  Then $\beta_1'$ bounds a disk (a bigon) $B$ in $A$ with that boundary component of $A$.   But, $\beta_1'$ forms a $d_1$-spiral ($d_1 \ge 1$) with $\gamma_1'$ in $A_1 \subset A$, which means that $\gamma_1'$ meets $\beta_1'$ at least 3 times in $A$.  Inside $B$ at most two of these intersections can be connected to the boundary $A$ by a subarc of $\gamma_1'$.  It follows that there must be another subarc of $\gamma_1' \cap B$ that has both endpoints on $\beta_1'$.  This forms a smaller bigon between $\beta_1'$ and $\gamma_1'$ inside $B$  and contradicts the assumption that $\beta$ and $\gamma$ have reduced intersection.   Thus $\beta_1'$ and (by a symmetric argument) $\gamma_1'$ are spanning arcs.

\begin{figure}
\begin{center}
\includegraphics{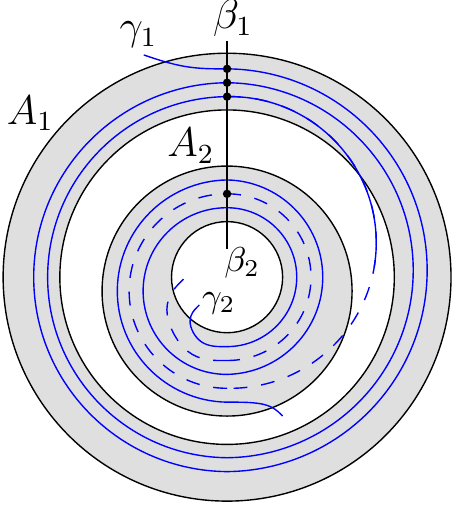}
\caption{Concentric $1$-spirals imply a $2$-spiral. In this figure $\beta_1$ connects to  $\beta_2$ in $A$ but $\gamma_1$ does not connect to $\gamma_2$ in $A$.}
\label{fConcentricSpirals}
\end{center}
\end{figure}

By assumption $\beta_1'$ and $\gamma_1'$ have at least $d_1+2$ intersections
in $A_1$.   Since they span $A$ they also pass through $A_2$, where they either
subsume or remain disjoint from the subarcs $\beta_2$ and $\gamma_2$, respectively. In the worst case neither is subsumed, but Lemma~\ref{lemLEQ2} still
guarantees that $\beta_1'$ and $\gamma_1'$ have at least $d_2+2 - 2 = d_2$
additional intersection in $A_2$. 
Therefore  $\beta_1'$ and $\gamma_1'$
have at least $d_1+d_2+2$ intersections in $A$, so the total curves $\beta$ and $\gamma$ form a $(d_1+d_2)$-spiral in
$A$ (and $F$).
\end{proof}

\begin{remark}\label{rem:counterwind}
The reader may wonder about the case when the concentric spirals wind in opposite directions, one clockwise and one counter-clockwise. In fact, this does not occur because the intersection is assumed to be reduced.  In the above proof, if the spirals wound in opposite directions, then a bigon would be formed whose corners were the last intersection between $\beta_1'$ and $\gamma_1'$ in $A_1$ and the first additional intersection between $\beta_1'$ and $\gamma_1'$  in $A_2$ (they have opposite algebraic sign).
\end{remark}

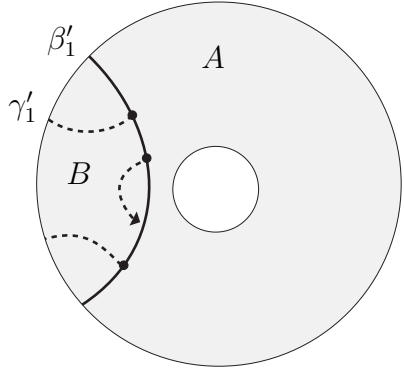
\begin{figure}
\begin{center}
\begin{overpic}[width=2in]
{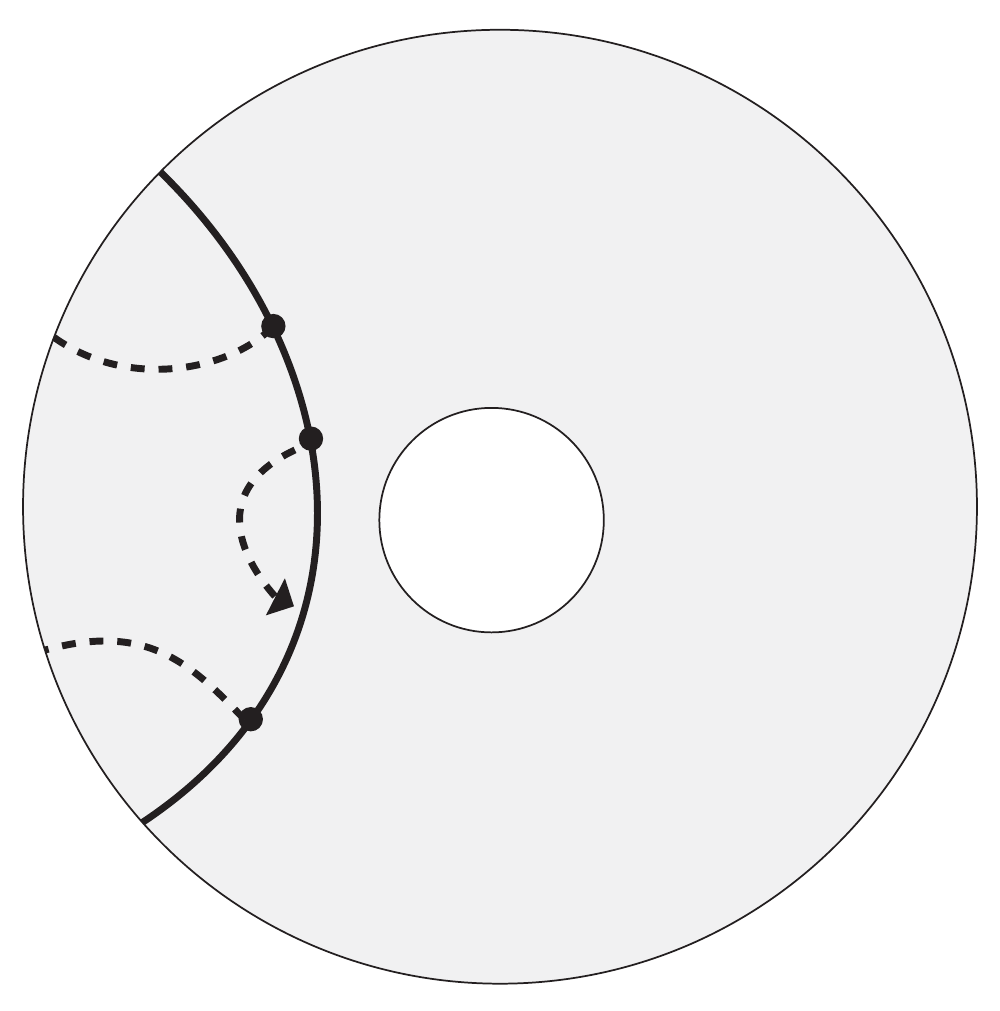}
\put(5,85){$\beta_1'$}
\put(45,80){$A$}
\put(10,50){$B$}
\put(-5,68){$\gamma_1'$}
\end{overpic}
\caption{If $\beta_1'$ does not span $A$, then the intersection between $\beta_1'$ and $\gamma_1'$ is not reduced.}
\label{fMustSpan}
\end{center}
\end{figure}

\subsection{Train Tracks}

Train tracks give us a shorthand for drawing complicated
multicurves on a surface.  The basic idea is captured in Figure~\ref{fTrainTrack}.  In a complicated multicurve it can be hard to differentiate individual subarcs that take a similar path. A weighted train track embraces their similarity by combining these subarcs and noting their count.

\begin{figure}
\begin{center}
\includegraphics[width=6in]{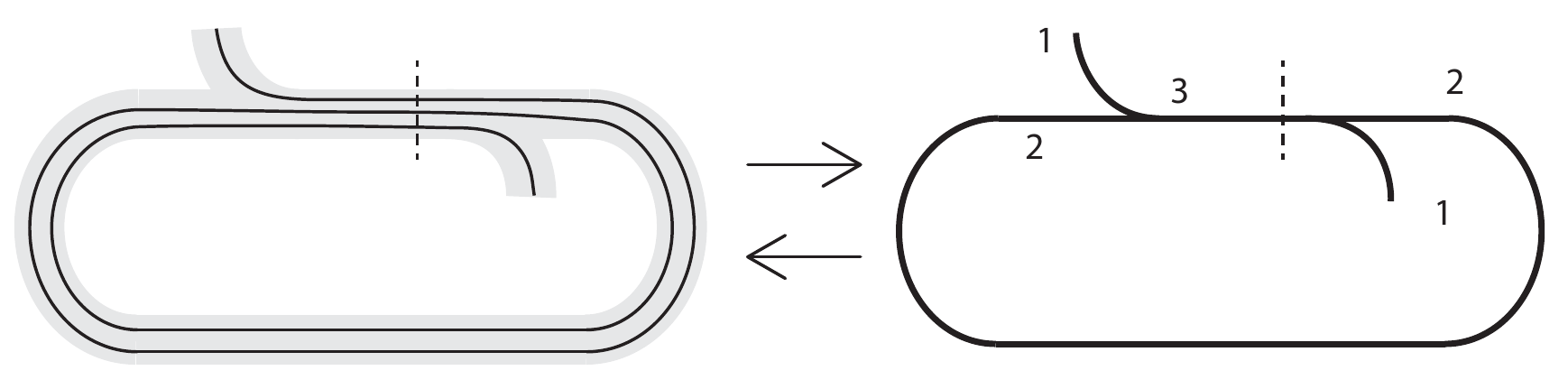}
\caption{A multicurve and the corresponding weighted train track.  The dashed arc is a basis arc for the branch with weight 3. That basis arc and the multicurve form a 1-spiral in a neighborhood of the track.}
\label{fTrainTrack}
\end{center}
\end{figure}

We now give a more formal definition, for a reference, see~\cite{PennerHarer}.  A \emph{train track} is a graph $T$ that is smoothly embedded in
a surface $F$. All vertices must have degree 1 or 3.  The vertices of degree one are called \emph{terminals},
the vertices of degree three are called \emph{switches}, and the
edges are called \emph{branches}. The embedding into $F$ must be smooth in the sense that: a) the interior of each branch/edge is a smooth curve, and b) at each switch, two of the branch ends meeting the switch are smooth extensions of the third. A \emph{basis curve} for a branch is a  simple curve (either an arc or closed curve) in $F$ that meets the train track transversely in a single point on that branch.   

We specify an embedded multicurve in an arbitrarily small
neighborhood of the track $T$ by assigning a \emph{weight} $w_i$, a
non-negative integer, to each branch.   The \emph{weight vector}
$\overrightarrow{w}$ will denote the collection of weights for each
branch.   The weighted track $\{T | \overrightarrow{w}\}$ determines
a unique (up to isotopy) multicurve 
in the neighborhood of the track, where the
weights indicate the number of times the multicurve travels along
each branch. We ensure that the multicurve respects the tangency
conditions at the switches by requiring that the weights at each switch
of $T$ satisfy a \emph{matching equation}: the sum of the weights of the branches
entering a switch is equal to the weight of the branch exiting. For example, in the track in
Figure~\ref{fTTTorus}, the two switches yield the same matching equation,
$p=q+r$. Note that we overload the notation, using a variable (such as $p,q,r$) to both label a branch as well as to indicate the weight of a multicurve on that branch.   We also adopt the convention that a
branch meeting a terminal has weight $1$, which means it represents  the endpoint of an arc in the resulting multicurve.
We say that $T$ \emph{carries} a multicurve $\beta \subset F$ if
there is a set of weights on $T$ that induces $\beta$ and satisfies the matching equations.

We say that a weighted train track $\{T | \overrightarrow{w}\}$ contains a 
\emph{$d$-spiral} if some component of the multicurve determined by the weights forms a
$d$-spiral in
a neighborhood of the track
with a basis arc for a  branch of $T$.
In Figure~\ref{fTrainTrack} the weighted train track contains a 1-spiral because the curve specified by the weights forms a 1-spiral with a basis arc that crosses the branch labeled 3.

\section{Avoiding Spirals on the Torus}
\label{sec:noSpiralsTorus}

In Section~\ref{sec:nospiraltorus} we construct a spiral-free drawing of two (reduced) curves on the torus with an arbitrary number of intersection points. As a preparatory step, Section~\ref{sec:TorusCurves} shows how to build a pair of closed curves avoiding $2$-spirals.

\subsection{Fibonacci Curves: 2-Spiral-Free Curves on the Torus}
\label{sec:TorusCurves}

Our goal is to build for any given $m$ a pair of closed curves on the torus that have at least $m$ intersections but that do not form a $d$-spiral for $d \ge 2$.

Fortunately, it is easy to represent essential closed curves on the torus, see for example Stillwell~\cite[Sections 6.2.2, 6.4.2, 6.4.3]{Stillwell93}. First, one chooses a (non-unique) \emph{basis}, which is a pair of curves, $\mu$ and $\lambda$ that intersect once. We can think of the longitude, $\lambda$, as the curve that goes the ``long'' way around the torus once and the meridian, $\mu$, as a curve that goes once around the ``short'' way. (This is a highly prejudicial view, but one that won't lead us astray).  Then any closed multicurve on the torus, $\gamma$, can be represented by an ordered pair of integers, measuring the number of times the curve wraps longitudinally and the number of times the curve wraps meridianally.

\begin{figure}
\begin{center}
\includegraphics{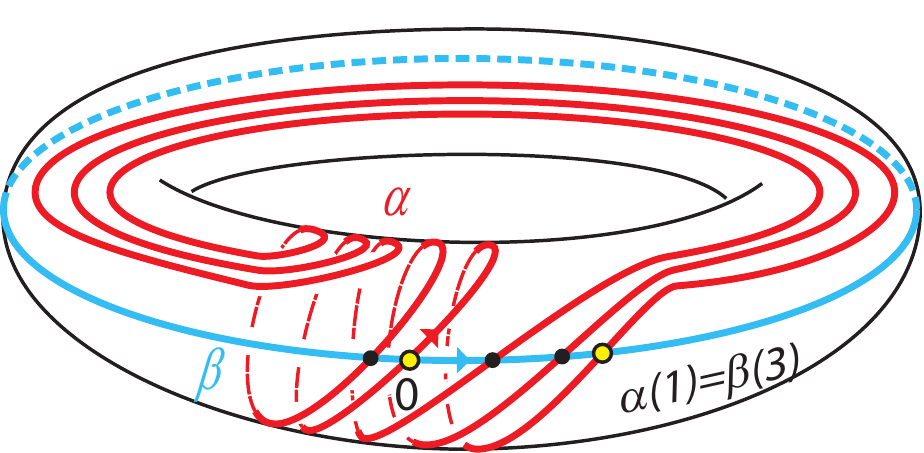}
\caption{Computing the shift, $q=3$. The curves $\alpha$ and $\beta$ intersect with pattern $(5,3)$.}
\label{fShiftTorus}
\end{center}
\end{figure}

Now, we classify intersection patterns between two closed curves in the torus, see Figure~\ref{fShiftTorus}. Let $\alpha$ and $\beta$ be oriented essential simple closed curves on the torus with reduced intersection consisting of $p$ points.  Choose one of the intersection points and label it as the origin $0$.  On each curve, start at the origin and travel along the curve in the direction of its orientation.  This induces a parameterization of each curve by the interval $[0,p)$, where the intersection points are $\{ \alpha(0)=0,\alpha(1),\dots,\alpha(p-1)\}$ along $\alpha$ and $\{ \beta(0)=0,\beta(1),\dots,\beta(p-1)\}$ along $\beta$. We write $\alpha( [a,b] )$ for the subcurve of $\alpha$ between points $\alpha(a)$ and $\alpha(b)$ (and similarly for $\beta$).
We say that the \emph{shift} of $\alpha$ with respect to $\beta$ is the value $q$ for which $\beta(q) = \alpha(1)$.  It is well known that, for curves with reduced intersection, the intersection pattern is transitive; that is, there is an orientation-preserving homeomorphism of the torus that preserves $\alpha$ and $\beta$ and takes any intersection point to any other.  It follows that the shift $q$ is independent of the choice of the origin $0$.

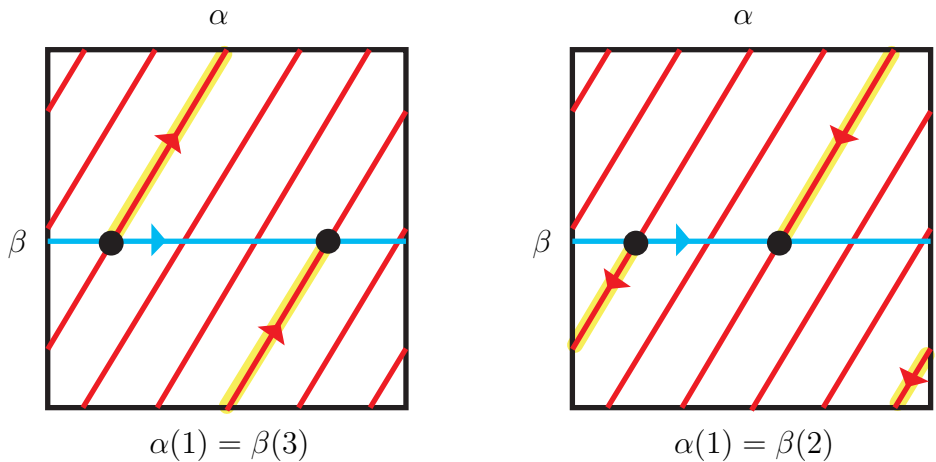
\begin{figure}
\begin{center}
\begin{overpic}
{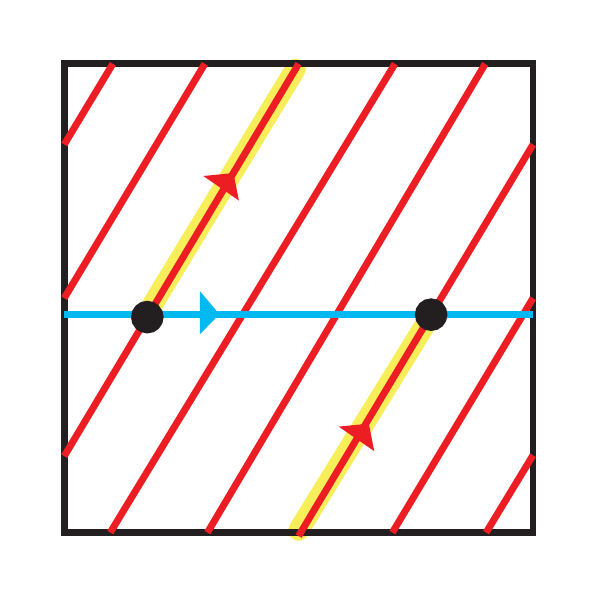}
\put(2,45){$\beta$}
\put(46,95){$\alpha$}
\put(33,1){$\alpha(1)=\beta(3)$}
\end{overpic}
\qquad
\begin{overpic}
{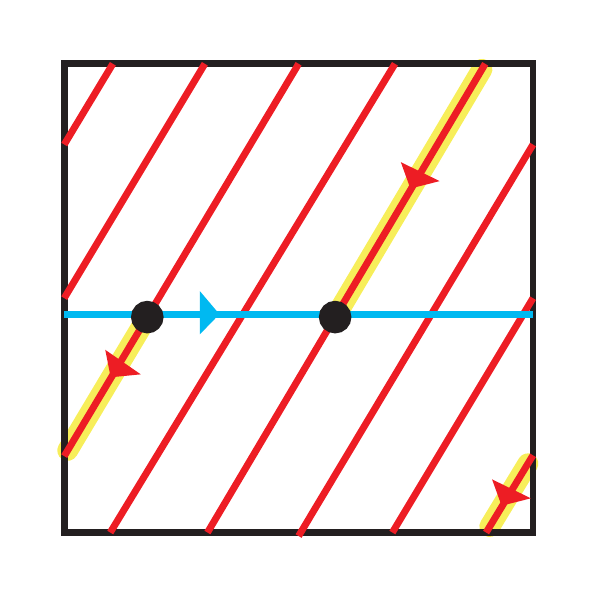}
\put(2,45){$\beta$}
\put(46,95){$\alpha$}
\put(33,1){$\alpha(1)=\beta(2)$}
\end{overpic}
\caption{Reversing $\alpha$ (or $\beta$) changes the intersection pattern from $(p,q)$ to $(p,p-q)$.   The highlight traces the curve to compute the shift.}
\label{fPattern}
\end{center}
\end{figure}

If $\alpha$ meets $\beta$ in $p$ points with shift $q$ then $p$ and $q$ are relatively prime and $0 \le q < p$, and we will say that $\alpha$ intersects $\beta$ with \emph{pattern} $(p,q)$.  
Also note that reversing the direction of $\alpha$ moves the intersection point $q$ units in the opposite direction along $\beta$, so the the resulting shift will be $p - q$.

We say that a $d$-spiral between oriented curves $\alpha$ and $\beta$ is \emph{positive} if the $d+2$ intersections that define the spiral are encountered in the same order when traveling positively along both $\alpha$ and $\beta$.

\begin{remark} It is not hard to see that if $\alpha$ and $\beta$ are essential simple closed curves on the torus whose reduced intersection consists of $3$ or more points, then $\alpha$ and $\beta$ form a $1$-spiral in the torus, since for any intersection pattern $(p,q)$ we have $\min(q,p-q)<p/2$.
\end{remark}

Let $F_0 = 0, F_1 = 1, F_2 = 1, \ldots$ be the sequence of Fibonacci numbers.  
The goal of this section is to prove Theorem~\ref{tFibsNo2Spiral} which says that a pair of simple closed curves on the torus with intersection pattern $(F_n,F_{n-1})$ do not form a 2-spiral.  The next lemma rules out the possibility of a \emph{negative} 2-spiral for, as  noted above, a \emph{negative} 2-spiral between curves  with pattern $(F_n,F_{n-1})$ is equivalent to a \emph{positive} 2-spiral between curves with pattern $(F_n,F_n-F_{n-1}) = (F_n,F_{n-2})$. The proof is by induction. After showing that there are no ``wide'' 2-spirals, we show that a positive 2-spiral between curves with pattern  $(F_n,F_{n-2})$ implies a positive 2-spiral with pattern $(F_{n-2},F_{n-4})$.

\begin{lemma}
\label{lemma_no_positive_2spiral} 
Let $\alpha$ and $\beta$ be simple closed curves on the torus with reduced intersection.
If $\alpha$ intersects $\beta$ with pattern $(F_n,F_{n-2})$, $n \ge 2$, then $\alpha$ and $\beta$ form no positive $2$-spiral.
\end{lemma}

\begin{proof}
We induct on $n$.   The result is trivially true for $n \le 4$ because $\alpha$ and $\beta$ have at most $F_4 = 3$ intersections, not enough to form a $2$-spiral which requires at least $4$ distinct intersections.

\begin{figure}
\begin{center}
\includegraphics[height=1.5in]{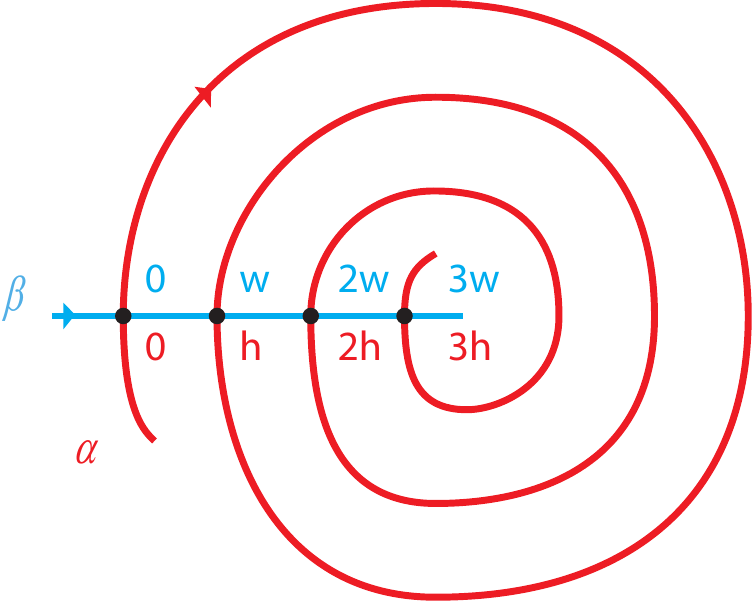}
\qquad
\includegraphics[height=1.5in]{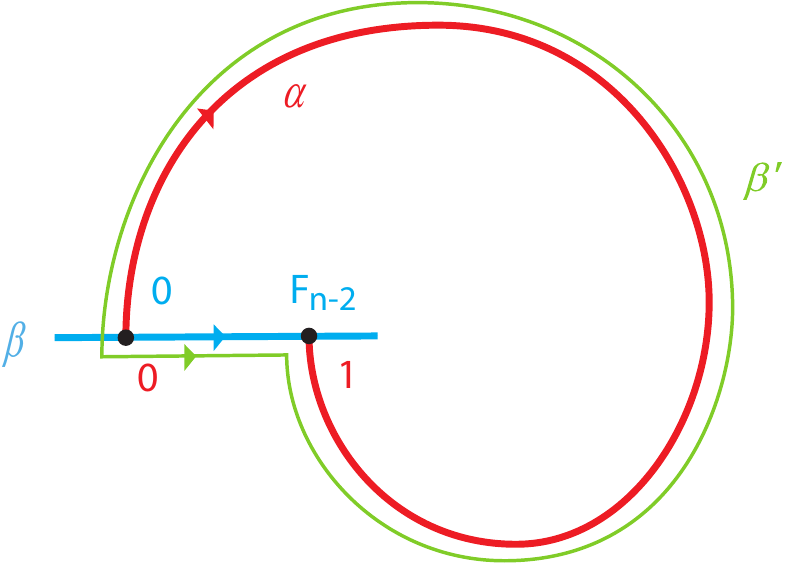}
\caption{A 2--spiral, forming $\beta'$.}
\label{fDoubleSpiral}
\end{center}
\end{figure}

Suppose that $n\ge 4$ and that $\alpha$ and $\beta$ form a positive $2$-spiral. See Figure~\ref{fDoubleSpiral}. The complement of the spiral has two rectangular regions.  Each rectangle has subarcs of $\alpha$ as its left and right sides and subarcs of $\beta$ as its top and bottom sides.  As the curves have reduced intersection, any other arc (not drawn) that crosses one side of a rectangle, say the left, must continue through the rectangle and cross the opposite, say right, side.  So parallel sides of the rectangles have the same number of intersections.  Thus we can 
define the \emph{width} and the \emph{height} of the spiral as the positive integers $w$ and $h$, respectively, such that the four intersections that define the spiral occur at positions $\beta(0) = \alpha(0) = 0, \beta(w)= \alpha(h), \beta(2w)= \alpha(2h)$ and $\beta(3w)=\alpha(3h)$.

First, a 2-spiral cannot be too wide, specifically we show that $F_{n-2} > 3w$:   Suppose to the contrary that $F_{n-2} \le 3w$.  Since  $\alpha(1) = \beta(F_{n-2})$, it must be that this point occurs somewhere in the interval $\beta([0,3w])$.  And as it is the first intersection along $\alpha$ it follows that $h=1$ and $w  = F_{n-2}$.  But then $3F_{n-2} = 3w < F_n = 2F_{n-2}+F_{n-3}$, a contradiction.

Since $F_{n-2} > 3w$, the arcs $\alpha([0,1])$ and $\beta([0,F_{n-2}])$ meet only in their endpoints and thus their union $\beta' = \alpha([0,1])  \cup \beta([0,F_{n-2}])$ is a simple closed curve.   Moreover if we slide $\beta'$ slightly in the negative direction with respect to $\beta$, see Figure~\ref{fDoubleSpiral}, it is a closed curve that meets $\alpha$ in $F_{n-2}$ points.   Every time $\alpha$ crosses $\beta'$ at one of these points it crosses in the same direction (same algebraic sign).  It follows both that  $\beta'$ is  essential and that $\alpha$ and $\beta'$ have reduced intersection. 

Now compute the shift of $\alpha$ with respect to $\beta'$:  All intersection points between  $\alpha$ and  $\beta'$ lie on the subarc $\beta([0,F_{n-2}-1])$ of $\beta'$.  So when we travel along $\alpha$, starting from $\alpha(0) = \beta'(0)$, we encounter intersections with $\beta$.  But only those in the interval $\beta([0,F_{n-2}-1])$ are also intersections with $\beta'$.   So when does $\alpha$ return to the interval $\beta([0,F_{n-2}-1])$? The first intersection of $\alpha$ with $\beta$ is the point $\alpha(1)=\beta(F_{n-2})$.  But this is just beyond the interval $\beta([0,F_{n-2}-1])$ so it is not an intersection with $\beta'$.  The second is at $\alpha(2)=\beta(2F_{n-2})$.  But, $2F_n-2 < F_n$, so this is still not an intersection with $\beta'$.  Finally, the third intersection is at the point $\alpha(3) = \beta(3F_{n-2})$.  As we should only consider the latter coordinate modulo $F_n$, we have
\[
3F_{n-2} \mod F_n = 3F_{n-2} - F_n = 3F_{n-2} - (2F_{n-2}+F_{n-3}) = F_{n-4}.
\]

So the shift of $\alpha$ with respect to $\beta'$ is $F_{n-4}$ and $\alpha$ and $\beta'$ intersect with pattern $(F_{n-2},F_{n-4})$.
Moreover, $\alpha$ and $\beta'$, because they have subarcs $\alpha([0,3h])$ and $\beta([0,3w])$ which form a positive 2-spiral, also form a positive 2-spiral.  This contradicts the inductive hypothesis.
\end{proof}

\begin{theorem}
\label{tFibsNo2Spiral}
Let $\alpha$ and $\beta$ be simple closed curves on the torus with reduced intersection. If $\alpha$ intersects $\beta$ with pattern $(F_n,F_{n-1})$, $n \ge 2$, then $\alpha$ and $\beta$ form no $2$-spiral.
\end{theorem}

\begin{proof}
First observe that by reorienting one of the curves we obtain curves intersecting with pattern $(F_n,F_n -F_{n-2})=(F_n,F_{n-2})$, which form no positive $2$-spiral by Lemma~\ref{lemma_no_positive_2spiral}, see also Figure \ref{fPattern}. This implies that $\alpha$ and $\beta$ form no negative $2$-spiral.

Suppose that $\alpha$ and $\beta$ form a positive $2$-spiral. We proceed similarly as in proof of Lemma~\ref{lemma_no_positive_2spiral}:   As before, the base cases $n =2,3$ hold because the curves do not have enough intersections.   For $n\ge 4$ we conclude that $F_{n-1} \ge 3w$ because $\beta(F_{n-1}) = \alpha(1)$.   Then we observe $\alpha$ and $\beta' = \beta([0,F_{n-1}]) \cup \alpha([0,1])$ have $F_{n-1}$ intersections.  The shift is computed by noting that $F_{n-1} < F_n < 2F_{n-1}$ so the curve ``rolls over'' after two original shifts, that is, the shift of $\alpha$ with respect to $\beta'$ is 
\[
2 F_{n-1} - F_n = F_{n-1} - F_{n-2} = F_{n-3}.
\]
Thus $\alpha$ and $\beta'$ intersect with pattern $(F_{n-1}, F_{n-3})$ and form a positive $2$-spiral.   This contradicts Lemma~\ref{lemma_no_positive_2spiral}.
\end{proof}

\subsection{Spiral-Free Curves on the Torus}\label{sec:nospiraltorus}

Figure~\ref{fTTTorus} shows two views of the same weighted train track $T_0$ on a
torus.  In the left hand picture
opposite sides of the square are
identified in order to form a torus.
Let $\mu$ be the curve formed by identifying the left and right sides
of the rectangle and oriented upwards, and $\lambda$ the curve formed by identifying the
top and bottom of the rectangle and oriented from left to right. 
Then $\mu$ and $\lambda$ intersect once and thus form a (homology) basis for the torus. Moreover, they also serve as basis curves for two branches of $T_0$. The train track $T_0$ carries all
closed multicurves on the torus whose coordinates $(p,q)$ satisfy $
p \ge 0, q \ge 0$; these coordinates are exactly the weights of the corresponding branches of the train track. (Recall that when considering un-oriented
curves, we may always assume $m\ge 0$.)

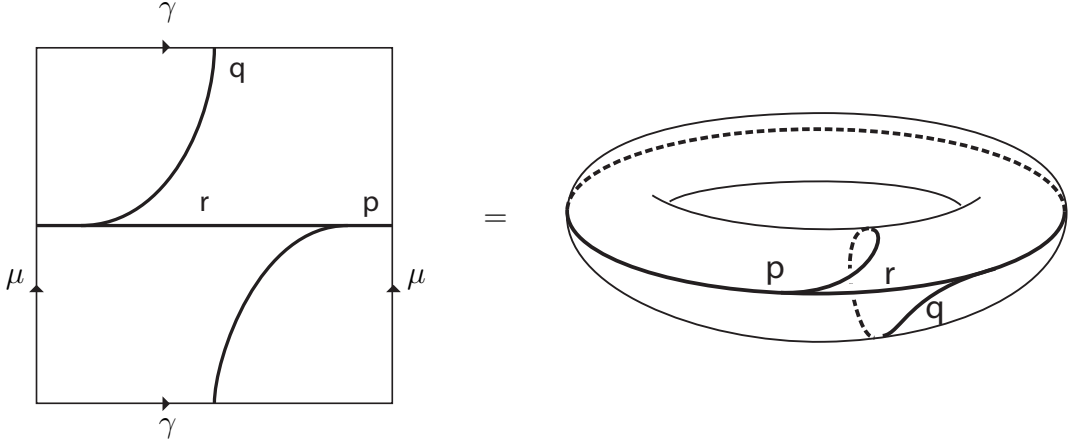
\begin{figure}
\begin{center}
\begin{overpic}[height=2in]
{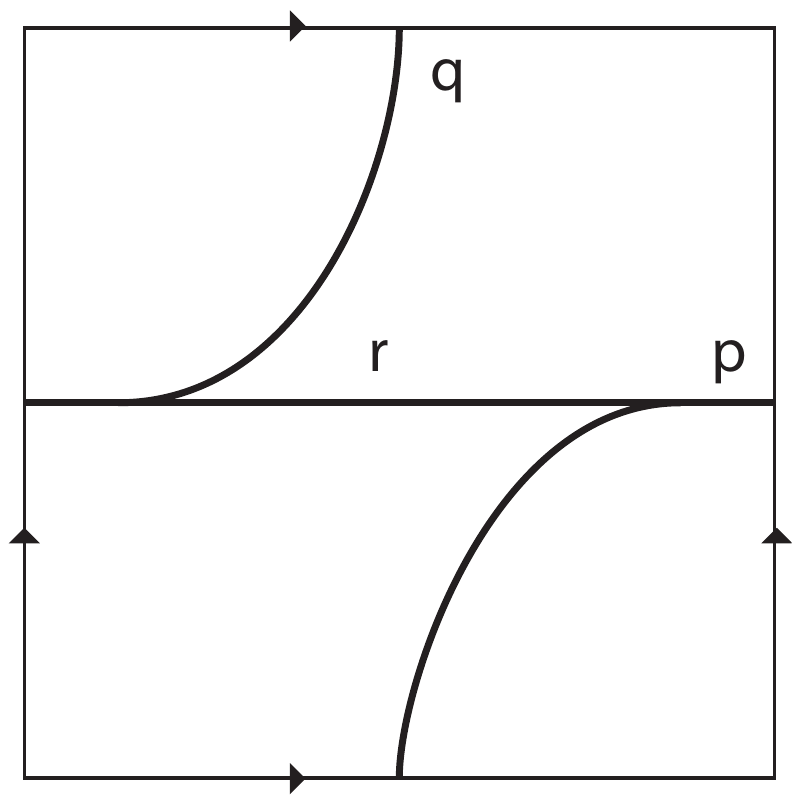}
\put(-5,35){$\mu$}
\put(100,35){$\mu$}
\put(35,-4){$\gamma$}
\put(35,105){$\gamma$}
\put(120,50){$=$}
\end{overpic}
\qquad \qquad
\begin{overpic}[height=2in]
{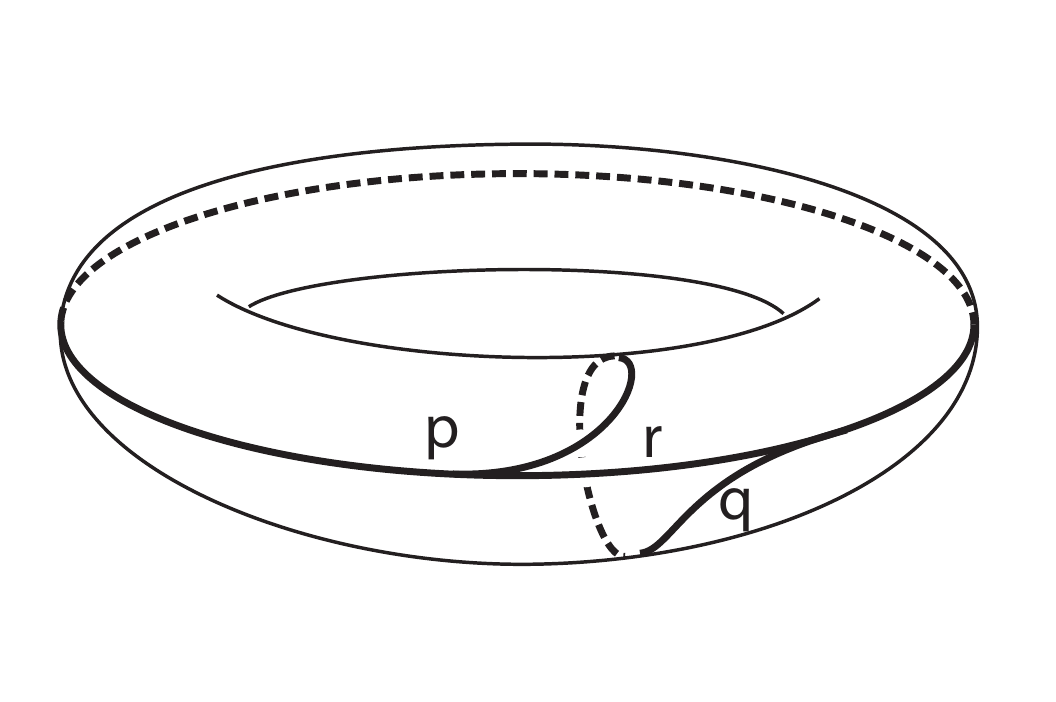}
\end{overpic}
\caption{Two views of the train track $T_0$ on a torus. The torus is formed by identifying opposite sides of the rectangle together, that is, left to right and top to bottom. The weighted train track $\{T_0 | p=F_n,q=F_{n-1},r=F_{n-2} \}$ contains no $d$-spiral for $d \ge 2$.} 
\label{fTTTorus}
\end{center}
\end{figure}

\begin{corollary}
\label{corNo2Spiral}
The weighted train track $\{T_0 | p=F_n,q=F_{n-1},r=F_{n-2} \}$ contains no $d$-spiral for $d \ge 2$.
\end{corollary}

\begin{proof}
	
Let $\gamma_n$ be the (connected) multicurve that lives in a neighborhood of $T_0$ and that is prescribed by the weights $p=F_n,q=F_{n-1},r=F_{n-2}$.  Recall that we say that the weighted train track contains a $2$-spiral if $\gamma_n$ forms a spiral with a basis arc of one of the branches, which are labeled $p$, $q$ and $r$. In fact, if the $2$-spiral uses the basis arc of $q$ or
$r$, then we can extend it to a $2$-spiral that uses the basis arc
of the branch labeled $p$: Suppose, for instance, that the spiral is formed with a basis arc of the, say, $r$ branch. Then extend that arc across the $q$ branch, and  slide it across the switch where the $p$ branch  meets  the $q$ and $r$ branches.  This yields a basis arc for the $p$ branch that defines a 2-spiral for the weighted track.

Thus if $T_0$ contains a $2$-spiral, then  $\gamma_n$ and $\mu$
form a $2$-spiral contradicting Theorem~\ref{tFibsNo2Spiral}, because $\gamma_n$ intersects $\mu$ with pattern
$(F_n,F_{n-1})$, the shift arising from the $F_{n-1}$ arcs that travel along the branch labeled $q$.
\end{proof}

Consider the train track $T_1$ pictured in Figure~\ref{fTT1}(a). We identify opposite edges of the square so this
train track is embedded in the torus.  In this section we will
demonstrate that $\{T_1| p=P=F_n-1, q=Q=F_{n-1}-1, r = R =
F_{n-2}-1,\text{terminals have weight } 1\}$ is spiral-free.  Note that the weights of some branches are not specified, but each such unspecified weight is uniquely determined as the the sum of the weights on the incoming branches.

\begin{figure}
\begin{center}

\begin{overpic}[height=1.75in]
{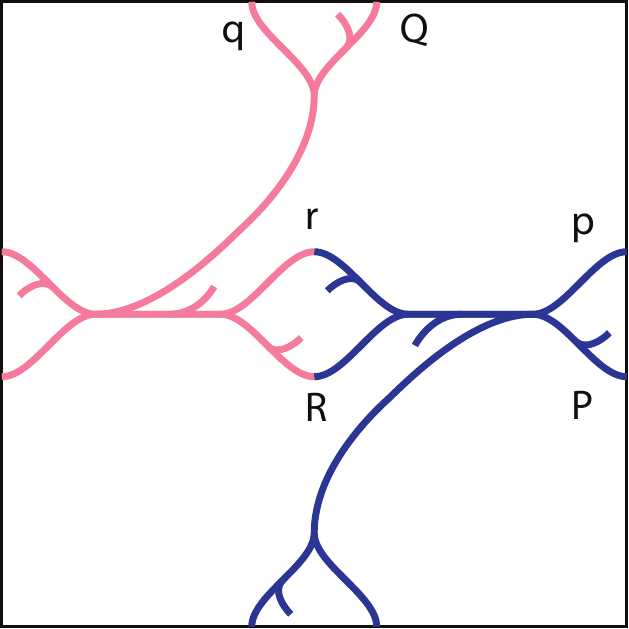}
\put(44,-13){(a)}
\end{overpic}
\qquad
\begin{overpic}[height=1.75in]
{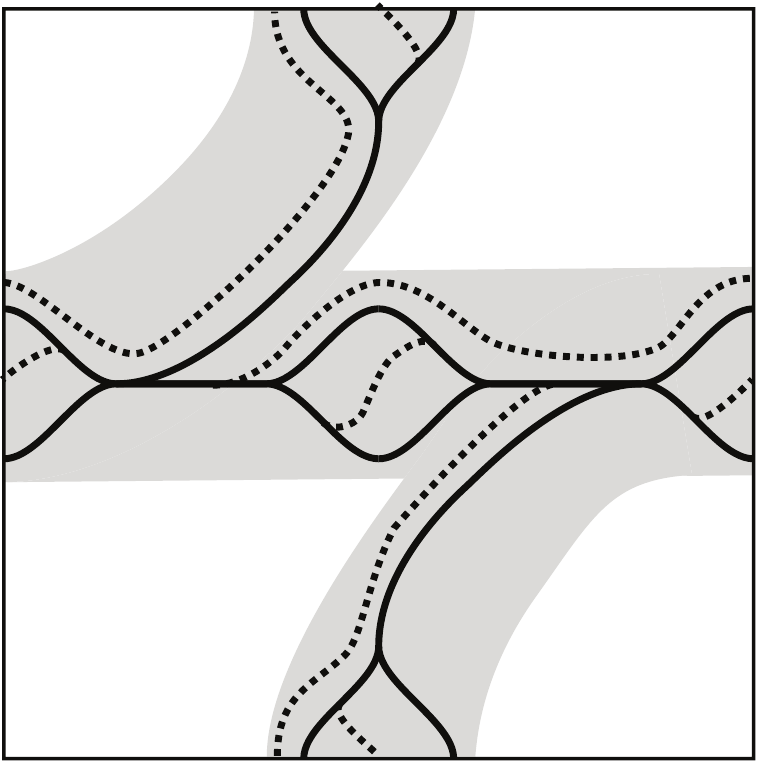}
\put(44,-13){(b)}
\end{overpic}
\qquad
\begin{overpic}[height=1.75in]
{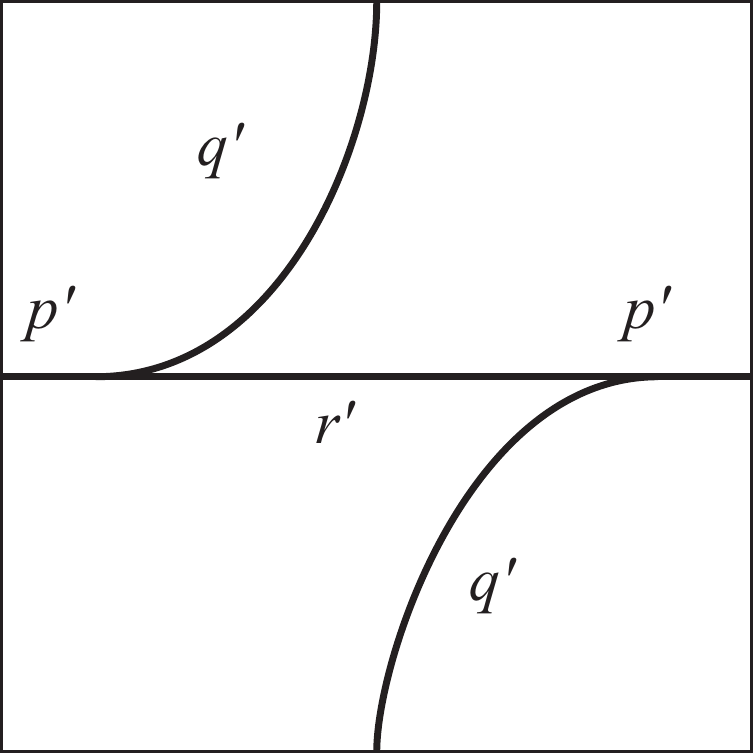}
\put(44,-13){(c)}
\end{overpic}
\vspace{.25in}

 \caption{a) The train track $T_1$.  The weighted train track  $\{T_1| p=P=F_n-1, q=Q=F_{n-1}-1, r = R =
 	F_{n-2}-1,\text{terminals have weight } 1\}$ has no $1$-spirals. b) Extending curves in $T_1$ to closed curves in $T_0$. c) The train track $T_0$.}
 \label{fTT1}
\end{center}
\end{figure}

The most interesting property of the track $T_1$ is that it is fixed under the self-homeomor\-phism of the torus,  call it $\pi$,  that is obtained by rotating the square about its center through the angle $\pi$.    Note that  $\pi$ is an involution, that is, it is its own inverse. Moreover, $\pi$ is orientation-preserving and acts on the branches of $T_1$ by changing the case of each branch, exchanging branches
with lower case labels with branches with upper case labels.  The only fixed points of the involution $\pi$ as it acts on the torus
are:  the center of the square; the corners of the square, all of which are identified to a single point; and the midpoint of each edge of the square
(opposite edges are identified).  Since the track $T_1$ avoids all of
these points, the restriction of $\pi$ to $T_1$ is fixed-point-free.
Fix a small neighborhood of $T_1$, $N(T_1)$,  that also avoids the fixed points, and is fixed setwise by $\pi$, so that $\pi:N(T_1) \rightarrow N(T_1)$ is a self-homeomorphism.  We can also choose short basis arcs for each branch of $T_1$ that are pairwise disjoint and such that their union is preserved by $\pi$ setwise. We highlight the following fact, which is established by the choice of $N(T_1)$.

\begin{lemma}
\label{lemFixedPoint}
The involution $\pi$ has no fixed point in $N(T_1)$. 
\end{lemma}

If any simple arc drawn in $N(T_1)$ is preserved setwise by $\pi$, then $\pi$ would have a fixed point in that arc, hence in $N(T_1)$. Lemma~\ref{lemFixedPoint} then implies the following.

\begin{corollary}
\label{cor_no_invariant_arc}
The involution $\pi$ preserves no simple arc carried by $T_1$.
\end{corollary}

Note that the train track $T_1$ can be naturally embedded inside a neighborhood of the train track $T_0$, see Figure~\ref{fTT1}(b).  Using this embedding we define an extension of any multicurve $\gamma$ carried by $T_1$ to a \emph{closed} multicurve $\epsilon(\gamma)$ carried by $T_0$ by connecting any endpoints of $\gamma$ at terminals of $T_1$ by the dotted arcs in the figure.   For closed curves $\epsilon$ does not modify the multicurve, and the resulting weights on $T_0$ will be the sum of the upper and lower case weights of $T_1$: $p'=p+P,q'=q+Q,r'=r+R$.   If $\gamma$ consists of
arcs where the weight of each terminal is $1$, then
the weights of the closed multicurve $\epsilon(\gamma)$ carried by
$T_0$ will be $p'=p+P+2,q'=q+Q+2,r'=r+R+2$.  The $+2$s count the extensions and are indicated by dotted lines in Figure~\ref{fTT1}(b).

The involution $\pi$ will preserve a multicurve carried by $T_1$ if
and only if the curve's weights are preserved by $\pi$, that is
$p=P,q=Q,r=R$, and terminals of equal weight are exchanged.

\begin{lemma}
\label{lemInvariantCurve}
The involution $\pi$ preserves no closed connected curve $\gamma$ carried by $T_1$.
\end{lemma}

\begin{proof} If $\gamma$ is preserved by $\pi$, then its weights must be
preserved by $\pi$, hence $p=P,q=Q,r=R$.  Then $\epsilon(\gamma)$ is
carried by $T_0$ and has even weights, each a sum of an equal
lowercase and uppercase weight. The curves $\gamma$ and $\mu$ then have intersection pattern $(2p,2q)$. Since $2p$ and $2q$ are not relatively prime and $\mu$ is connected, the curve $\gamma$ cannot be connected; see the discussion of intersection patterns in Section~\ref{sec:TorusCurves}.
\end{proof}

\begin{lemma}
\label{lemIsotopic}
 Let $\gamma$ be a closed curve carried by $T_1$.
Then $\epsilon(\gamma)$ and $\epsilon(\pi(\gamma))$ have the same weights (and hence are isotopic) in $T_0$.
\end{lemma}

\begin{proof} The involution changes the case of each of the weights of
$\gamma$, then  $\epsilon$ sums each lower and upper case weight. Thus the curves
$\epsilon(\gamma)$ and $\epsilon(\pi(\gamma))$ have equal weights
and are hence isotopic in $T_0$. \end{proof}

When $p=P,q=Q, r=R$, then the involution $\pi$ fixes the weight vector $\overrightarrow{w}$ as well as the track $T_1$, and we will say that $\{T_1 | \overrightarrow{w}\}$ is \emph{invariant under $\pi$.}  
Recall that a weighted train track is really a shorthand for a multicurve drawn in a small neighborhood of the track.  The weight component for each branch specifies the number of times the multicurve travels along that branch. If $\{T_1 | \overrightarrow{w}\}$ is invariant under $\pi$, then we will assume that the multicurve has been drawn carefully in $N(T_1)$ so that it is fixed setwise by $\pi$.   
We will say that $\{T_1 | \overrightarrow{w}\}$ contains two disjoint spirals if two disjoint spirals are formed in a neighborhood of $T_1$ by the multicurve prescribed by $\overrightarrow{w}$ and some basis arc for a branch of $T_1$.

\begin{proposition}
\label{pTwoSpirals}
 Suppose that the weighted train track $\{T_1 | \overrightarrow{w}\}$ is invariant under
$\pi$ and contains a spiral.  Then $\{T_1 | \overrightarrow{w}\}$
contains two disjoint spirals.
\end{proposition}

This result is based on analyzing the rectangle cut out by the spiral, see Figure~\ref{fSpiral}, and how the rectangle meets its image under the involution $\pi$.  In particular, if the rectangle is not disjoint from its image under $\pi$, then we produce a fixed curve, contradicting earlier lemmas.  It will follow that these two rectangles, and their corresponding spirals, are disjoint.

\begin{proof}

Recall that, since the weighted train track $\{T_1 | \overrightarrow{w}\}$ is invariant under $\pi$,  we have constructed the following objects, each of which is fixed setwise but contains no fixed point under the action of 
$\pi$: a) a neighborhood $N(T_1)$ of $T_1$, b) a collection of disjoint basis arcs, one for each branch, and c) the multicurve prescribed by $\overrightarrow{w}$.

Suppose that $\{T_1 |  \overrightarrow{w}\}$ does contain a spiral.   Then  $\gamma$, a subarc of the multicurve  prescribed by $\overrightarrow{w}$, and $\beta$, a subarc of a basis arc, bound a rectangle $R$ that meets itself in a single point as in Figure~\ref{fSpiral}. 
Without loss of generality, we may assume that the rectangle $R$ is oriented as in Figure~\ref{fRectangleIntersections}; namely that $b$, the point of self-contact, occurs at the upper-right and lower-left corners of $R$, and $a$ and $c$ are the top-left and bottom-right corners of $R$, respectively. 
The top and bottom sides of the rectangle are subarcs of $\gamma$ that connect at $b$ and the left and right sides are subarcs of the basis arc $\beta$ that also connect at $b$. Note that the entire drawing is reduced in $T_1$, because no
component of the multicurve can form a bigon with a basis arc.

We now consider the possible patterns of intersection $R \cap
\pi(R)$ with the aim of showing that it is empty, see Figure~\ref{fRectangleIntersections}.    Since the
involution $\pi$ exchanges $R$ and $\pi(R)$, the pattern of
intersection appears precisely the same in each rectangle.    The rectangle $R$ will be cut across vertically by some subarcs of basis arcs, and horizontally by some subarcs of the multicurve specified by the weight vector.   We can measure the \emph{length} of a subarc of the multicurve or basis arc to be the number of intersections it has with the basis arcs or multicurve, respectively. Since the drawing is reduced,  the lengths of the left and the right side of $R$ are the same, and thus we can define the \emph{height} of the rectangle $R$ as the length of its left side. Similarly we can define the \emph{width} of $R$ as the length of its bottom side.   Since  $\pi$ preserves both the collection of basis arcs and the multicurve, this pattern appears precisely the same in $\pi(R)$, and $\pi$ can be regarded as an isometry, as it preserves the length of any arc.   We study the intersection  $R \cap \pi(R)$ through a sequence of claims, and ultimately conclude that the intersection is empty.

\begin{figure}
\begin{center}
\includegraphics[width=\textwidth]{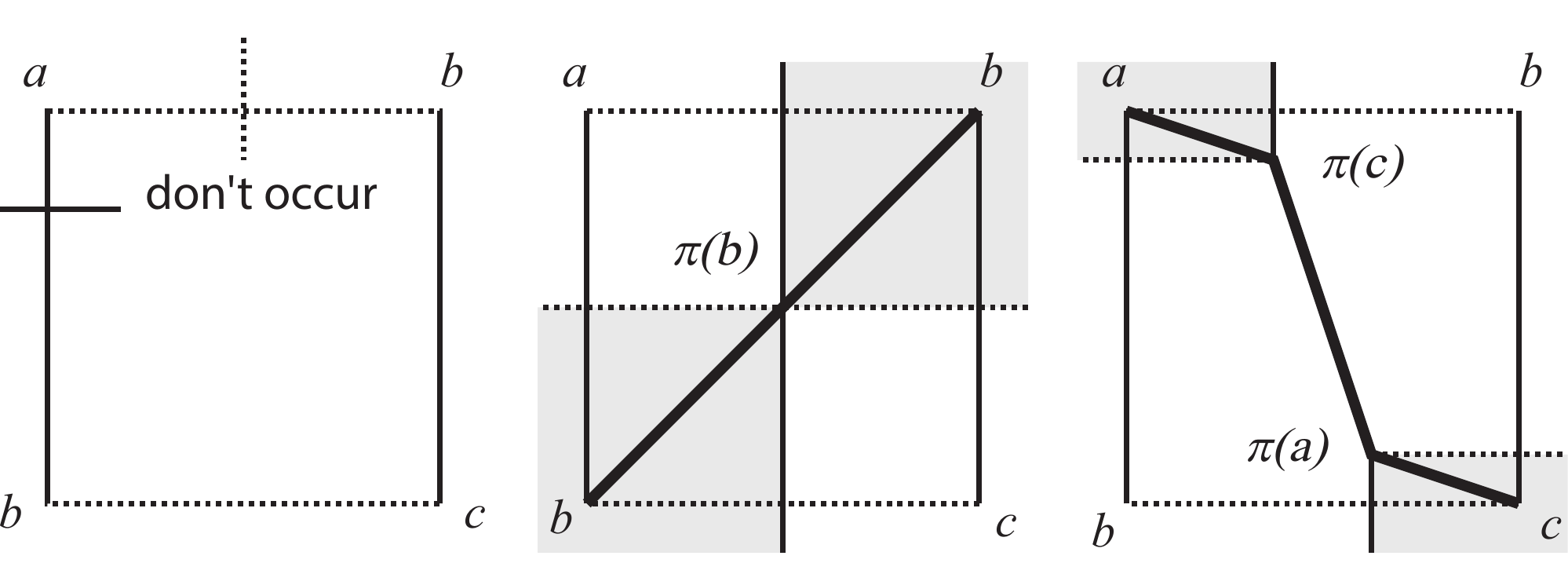}
 \caption{Intersections between $R$ and $\pi(R)$.  The vertical arcs are basis arcs ($\beta$), the horizontal arcs are portions of the multicurve ($\gamma$).  The shaded region indicates  $\pi(R)$.  }
 \label{fRectangleIntersections}
\end{center}
\end{figure}


\begin{claim_}
$\beta \cap \pi(\beta) = \emptyset$.  
\end{claim_}
Under the action of $\pi$, every basis arc of the track is moved to a disjoint basis arc.  The arc $\beta$ is a subarc of a basis arc so the claim follows.

\begin{claim_}
$\gamma \cap \pi(\gamma) = \emptyset$.    
\end{claim_}
Let $\gamma'$ be the component of the multicurve prescribed by $\overrightarrow{w}$ and containing $\gamma$. If $\gamma'$ is a closed curve, then $\gamma' \cap \pi(\gamma') = \emptyset$ by Lemma~\ref{lemInvariantCurve}. If $\gamma'$ is a simple arc, then $\gamma' \cap \pi(\gamma') = \emptyset$ by Corollary~\ref{cor_no_invariant_arc}. 

\begin{claim_}
$R \cap \pi(R)$ is the union of rectangles, each containing, in opposite corners,  a corner of $R$ and a corner of $\pi(R)$.   
\end{claim_}
Since $\pi$ preserves both the set of basis arcs and the multicurve, the overlap consists of a union of rectangles that are bounded by boundary arcs of $R$ and $\pi(R)$. In each component subrectangle, either the top is a subarc of $R$ and the bottom a subarc of $\pi(R)$, or vice-versa, because the alternative would imply that the height of $R$ is less than the height of $\pi(R)$, or vice-versa. Similarly, the left side is from $R$ and right side from $\pi(R)$, or vice-versa.  It follows that every subrectangle contains both a corner of $R$ and a corner of $\pi(R)$.

\begin{claim_}
$\pi(b) \notin R$.  
\end{claim_}
Suppose that $\pi(b) \in R$. By prior claims $\pi(\gamma)$ and $\pi(\beta)$ cut $R$ into quadrants and $R \cap \pi(R)$ includes either the two quadrants containing $a$ and $c$ or the two quadrants containing $b$. Since $\pi$ is an orientation-preserving map on the torus, the point $\pi(b)$ is both a bottom left and a top right vertex of a quadrant in $R \cap \pi(R)$, hence these two quadrants must contain the $b$ corners of $R$. See Figure~\ref{fRectangleIntersections}, middle picture. 
Now, $\pi$ must exchange these quadrants, for otherwise it would send one to itself and this would imply that $\pi$ has a fixed point in that quadrant, by the Brouwer's fixed-point theorem for a disk.  In one quadrant choose a diagonal joining $b$ to $\pi(b)$.  Then $\pi$ applied to that diagonal is a diagonal in the other quadrant joining $\pi(b)$ to $b$, and the union of the diagonal and its image is a closed connected curve that is carried by $T_1$ and is invariant under $\pi$. This contradicts Lemma~\ref{lemInvariantCurve}.

\begin{claim_}
$\pi(a) \notin R$ and $\pi(c) \notin R$.   
\end{claim_}
Suppose
$\pi(a) \in R$.   Then the $\pi(a)$ corner of $\pi(R)$ overlaps a corner
of $R$.   It cannot overlap a $b$ corner, for this would imply that
$\pi(b) \in R$, contradicting the previous claim.   Nor can it
overlap the $a$ corner.  For in this case we would see a rectangle of
intersection between the top-left corners of $R$ and $\pi(R)$.   But
such a rectangle of intersection would be invariant under $\pi$, and
would contain a fixed point.  So we consider what happens if the $\pi(a)$
corner of $\pi(R)$ overlaps the $c$ corner of $R$. Because the
intersection pattern is symmetric, we also see the $\pi(c)$ corner of
$\pi(R)$ overlapping the $a$ corner of $R$. In $R$ choose two connecting arcs, one from $a$ to $\pi(c)$ and the other from $\pi(c)$ to $\pi(a)$, see rightmost image in Figure~\ref{fRectangleIntersections}.  Then $\pi$ applied to this pair of connecting arcs is another pair of connecting arcs in $\pi(R)$, the first from $\pi(a)$ to $c$ and the other from $c$ to $a$ (not pictured).    The union of the four arcs is a closed connected curve that is carried
by $T_1$ and invariant under $\pi$, again in contradiction to Lemma~\ref{lemInvariantCurve}.   The same argument proves $\pi(c) \notin R$.

Since an overlap would imply that a corner of $\pi(R)$ is contained in $R$, and this does not occur, we conclude that the rectangles $R$ and $\pi(R)$ are disjoint. A slight neighborhood of
$R$ is a spiral annulus $A$ that is disjoint from the spiral annulus
$\pi(A)$.   This proves the proposition.
\end{proof}

\begin{theorem}
\label{tT1SpiralFree} The weighted track
 $\{T_1| p=P=F_n-1, q=Q=F_{n-1}-1, r = R = F_{n-2}-1,\text{terminals } \allowbreak \text{have weight } 1\}$ contains no $d$-spiral for $d\ge 1$.
\end{theorem}

\begin{proof}
If there is a $d$-spiral with $d \ge 1$, then there is a $1$-spiral, and Proposition~\ref{pTwoSpirals}
implies that we can actually find two disjoint 1-spirals.
By Lemma~\ref{lemIsotopic}, $\epsilon$ maps these to a pair of disjoint parallel spirals in $\{T_0 | p'=2F_n,q=2F_{n-1},r=2F_{n-2}\}$.   

But then, Lemma~\ref{lemConcentricSpirals} implies that $\beta$, a basis arc, and $\alpha$, a component of the multicurve prescribed by the weights, form a $2$-spiral in $\{T_0 | p'=2F_n,q'=2F_{n-1},r'=2F_{n-2}\}$.  (Note that Lemma~\ref{lemConcentricSpirals} does not require that $\alpha$ and $\alpha'$ are subarcs of the same multicurve, only that the two spirals are disjoint, which forces $\alpha$ through the $\alpha'$ spiral and vice-versa.)

Now the multicurve $\{T_0 | p'=2F_n,q'=2F_{n-1},r'=2F_{n-2}\}$ has two components, each with equal weights.  Hence $\alpha$ is described by the weighted track $\{T_0 | p'=F_n,q'=F_{n-1},r'=F_{n-2}\}$.    But this is in direct contradiction to Corollary~\ref{corNo2Spiral} which states, in particular, that this weighted track contains no 2-spiral.  
\end{proof}

\section{Avoiding Spirals in the Plane}
\label{sec:noSpiralsPlane}

In Section~\ref{sec:SFCP} we construct a spiral-free drawing of two (reduced) curves in the plane with an arbitrarily large number of intersection points. This allows us to give, in Section~\ref{sec:counter}, a counterexample to a proof by Pach and T\'{o}th~\cite{PT01} that the number of intersections in a minimal realization of a string graph is at most exponential.

\subsection{Spiral-Free Curves in the Plane}
\label{sec:SFCP}

We are now ready to prove our ultimate goal, that the planar
weighted track $\{T _2|  p=F_n-1, q
=F_{n-1}-1,r=F_{n-2}-1,\text{terminals have weight } 1\}$, pictured in Figure~\ref{fT2}, is
spiral-free.
Our proof exploits the fact that the weighted track $T_1$ is a \emph{double covering} of the weighted $T_2$ track.  If the weighted $T_2$ track contains a spiral, we will demonstrate that it can be either be \emph{lifted} (pulled back) to $T_1$, which contradicts Proposition \ref{pTwoSpirals}, or that the weighted track $T_1$ contains a fixed curve, which contradicts Lemma \ref{lemInvariantCurve}.

\begin{figure}
\begin{center}
\includegraphics{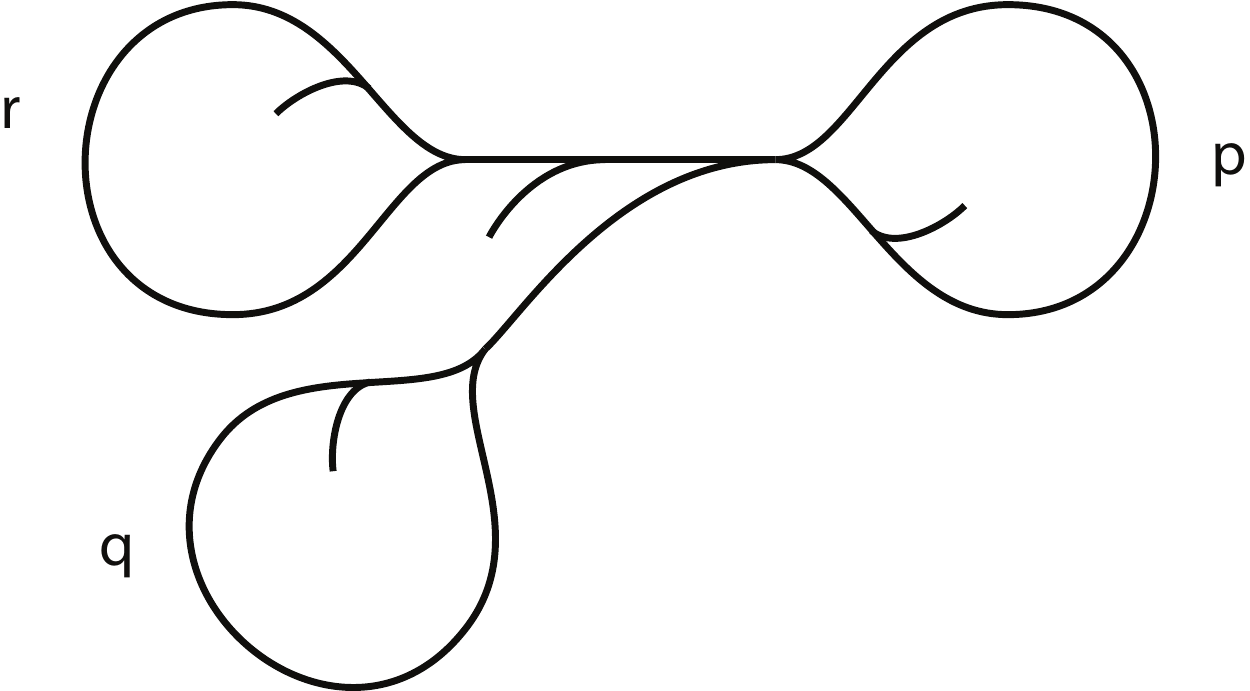}
\caption{The track $T_2$. The weighted track $\{T_2|
p=F_n-1, q =F_{n-1}-1,r=F_{n-2}-1, \text{terminals have weight } 1\}$
has no spirals.}
\label{fT2}
\end{center}

\end{figure}

There is a strong similarity between the (weighted) tracks $T_1$ and
$T_2$, a similarity induced by  $\pi$. We can view $T_1$ as
formed by cutting each of the loops of $T_2$, taking two copies, and
then gluing them together. Or, alternatively we can see that $T_2$ is a quotient of $T_1$ under the action of $\pi$.  Let $N(T_1)$ be a regular neighborhood of $T_1$ that is invariant under the action of $\pi$, but for which no point is fixed by $\pi$.  Let $N(T_2)$ be the \emph{quotient space} $N(T_1)/(x \sim \pi(x))$, where each each point has been identified with
its image under $\pi$. Then, there is a natural projection $\rho: N(T_1) \rightarrow N(T_2)$ sending each point in $N(T_1)$ to its corresponding identified point:  $\rho(x) = \rho(\pi(x))$. In fact, $\rho$ is a \emph{local homeomorphism}, and restricts to a
homeomorphism of any neighborhood $U \subset N(T_1)$ that is disjoint
from $\pi(U) \subset N(T_1)$. Together, the space $N(T_1)$ and the
map $\rho$ are referred to as a \emph{covering space} of $N(T_2)$; see~\cite[Chapter 1.3]{Hatcher}. Covering spaces have very strong properties. Here we
will use only the fact that any simply-connected region $U \subset
N(T_2)$ \emph{lifts to $N(T_1)$}, that is, the inverse image $\rho^{-1}(U)$
consists of $n$ disjoint regions $\{U_1, U_2, \dots,U_n\} \subset
N(T_1)$, each homeomorphic to $U$. In our case, $N(T_1)$ is a double-cover
of $N(T_2)$, that is, $n=2$. Points, arcs, disks and embedded rectangles
are all simply connected, so each will lift to two disjoint
copies in $N(T_1)$.

Because the lower case and upper case weights are equal, $\rho$ can
also be regarded as a map between the weighted tracks $\rho:\{T_1|
p=P=F_n-1, q=Q=F_{n-1}-1, r = R = F_{n-2}-1,\text{terminals have weight } 1\} \rightarrow \{T_2| p=F_n-1, q
=F_{n-1}-1,r=F_{n-2}-1,\text{terminals have weight } 1\}$ that converts to lower case and projects induced curves onto induced curves. In fact, though we will not prove
it here, the multicurve specified by these weights in $N(T_2)$ is a union of two properly
embedded arcs (contains no closed curves), and therefore lifts to the union of precisely four properly embedded arcs in $N(T_1)$.

There is one essential difference between $T_1$ and $T_2$: $T_2$ is
planar whereas $T_1$ can only embed in surfaces of genus at least $1$.

\begin{theorem}
\label{tT2SpiralFree}
 The weighted train track $\{T_2 | p=F_n-1, q =
F_{n-1}-1,r=F_{n-2}-1,\text{terminals have weight } 1\}$ contains no spiral.
\end{theorem}

\begin{proof} By way of contradiction, assume there is a spiral, thus there
is a rectangle with a single point of intersection $b$ formed by the induced multicurve and a basis arc in the neighborhood of $T_2$.    We cannot
lift this rectangle since the point of self-contact prevents it from
being simply connected.   In an abuse of notation we will let $R$
denote the rectangle with only the point $b$ removed, that is, a
rectangle with two open corners.   This $R$ is simply connected and
therefore lifts to two ``rectangles'', $R_1$ and $R_2$ in the double cover, a neighborhood of the weighted track $\{T_1|
p=P=F_n-1, q=Q=F_{n-1}-1, r = R = F_{n-2}-1,\text{terminals have weight } 1\}$.    There are also two
lifts of $b$: $b_1$ and $b_2$.   There are two possibilities: 1) each
of the rectangles $R_1$ and $R_2$ uses only one of the points $b_1$
and $b_2$, or 2) $R_1$ and $R_2$ are chained together, each using
both $b_1$ and $b_2$.  See Figure~\ref{fOwnShare}. In the former
case, $R_1$ and $R_2$ are both spirals, and either one can be used
to contradict Theorem~\ref{tT1SpiralFree}.

\begin{figure}
\begin{center}
\includegraphics{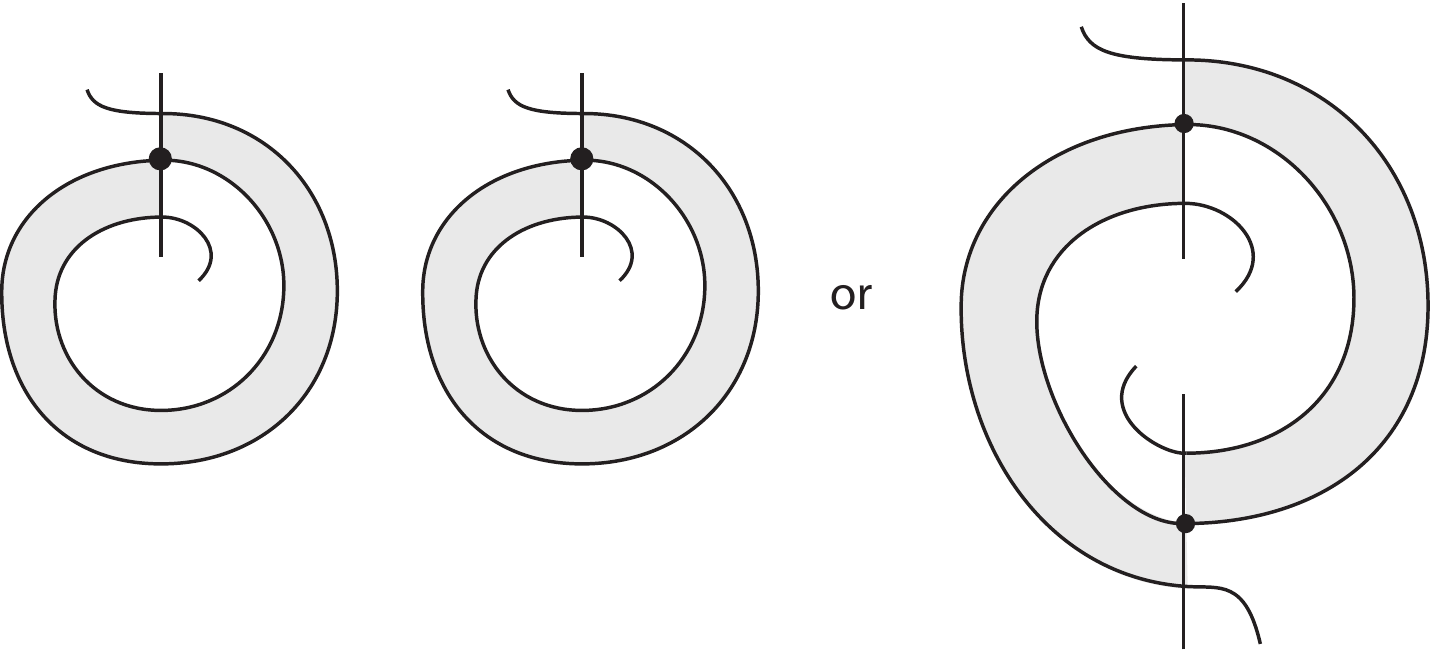}
\caption{The rectangles in $N(T_1)$ can either each have their own $b$
or they share the two $b$'s.} 
\label{fOwnShare}
\end{center}
\end{figure}

The latter case can also be seen to be a contradiction.    While the
involution $\pi$ exchanges $R_1$ and $R_2$, it preserves the union $R_1 \cup
R_2 \cup b_1 \cup b_2$.   Let $\beta_1$ be the diagonal arc of $R_1$ connecting $b_1$ and $b_2$, and $\beta_2 = \pi(\beta_1)$ the diagonal of $R_2$.   But then $\pi$ preserves the closed curve $\beta = \beta_1 \cup \beta_2$, contradicting Lemma~\ref{lemInvariantCurve}.
\end{proof}

\subsection{A Counterexample to Pach and T\'{o}th}\label{sec:counter}

Pach and T\'oth~\cite{PT01} proved that for every $n\ge 1$, every pair of curves in the plane with reduced intersection and with sufficiently many intersections, forms either a spiral of depth $n$ or a fold of width $n$. Their paper claims to show that in a minimal realization of a string graph both the depth of a spiral and the width of a fold are bounded above in the number of curves. An upper bound on the number of intersections between each pair of curves, hence of the entire string graph realization would follow. Their argument against spirals is correct: sufficiently deep spirals can be used to reduce the number of intersections in a realization, contradicting the minimality of the realization.

However, their argument that the existence of a wide fold implies the existence of a deep spiral is unfortunately incorrect and cannot be fixed easily.

The main idea in the proof is to start with the original wide fold, and assuming there are no deep spirals, construct a sequence of ``iterated'' folds of smaller widths, and in each iteration find an obstruction that prevents easy simplification of the drawing. The obstruction in each iteration is a pair $(g,v)$ where $g$ is an edge ``cutting all the way through'' the fold and terminating in a vertex $v$ inside the fold; see Figure~\ref{cutting_all_the_way}.
It is then argued that for each iteration a new obstruction is needed which would imply that after at most $2m+1$ iterations the drawing can be simplified. This claim, however, is false. We construct a counterexample containing a pair of curves forming an arbitrarily wide fold and having no spirals or bigons, with only four possible pairs $(g,v)$ that can serve as obstructions.

\begin{figure}
\begin{center}
\includegraphics{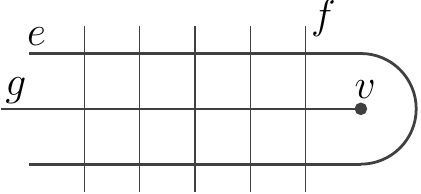}
\end{center}
\caption{A curve $g$ cutting all the way through a fold formed by $e$ and $f$.  In \cite{PT01}, this is referred to as an empty path of type 2.}
\label{cutting_all_the_way}
\end{figure}

The counterexample consists of two edges $e_1, e_2$ that are carried by the train track $T_2$ in Figure~\ref{fT2}, and an edge $f$ that crosses the train track transversely in the central part. We also add four auxiliary edges joined to the endpoints of $e_1$ and $e_2$. See Figure~\ref{counter_curves} for the case $n=6$.

\begin{figure}
\begin{center}
\includegraphics{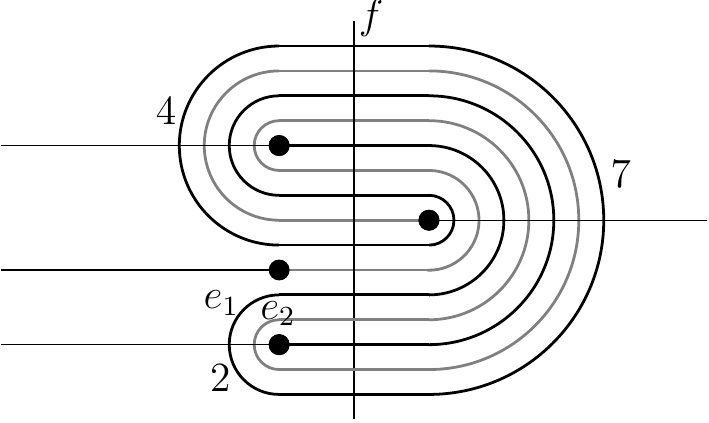}
\end{center}
\caption{A system of curves carried by the train track $T_2$ on Fig.~\ref{fT2} for $n=6$, extended to a drawing of a graph with $7$ edges.}
\label{counter_curves}
\end{figure}

By Theorem~\ref{tT2SpiralFree} the curves $e_1$ and $f$ do not form a spiral. Therefore, they form an arbitrarily wide fold if $n$ is arbitrarily large (see~\cite{PT01}, also~\cite{SSS11}). However, only the pairs $(v,e_i)$ where $v$ is an endpoint of $e_i$ are available as obstructions to simplifying the folds, so the number of obstructions is fixed and does not increase with the number of intersections.

\smallskip

This gap in Pach and T\'{o}th's argument cannot be filled easily. In their argument, folds are simplified only by local redrawings of the pairs of curves $g,h$ that form an empty bigon, or a bigon which may contain some vertices but all edges that cross the part of the boundary formed by $g$ also cross the opposite part of the boundary formed by $h$. In our example, no such bigons exist, and there are no spirals, so none of their redrawing techniques apply. 

It {\em is} possible to simplify the drawing in our example by a simultaneous local redrawing of both edges $e_1$ and $e_2$. However, it is not clear how to generalize this operation for systems of more than two curves. The related \emph{Simple shortcut problem}, introduced by {\v{S}}tefankovi{\v{c}}~\cite[Section 4.2]{S05}, asks about the possibility of simplifying a drawing of a graph in the neighborhood of one edge, and it is still wide open.

\bibliographystyle{vlastni}
\bibliography{noSpiralsPlane_arxiv}

\end{document}